\documentclass[11pt]{article} 
\usepackage{amsmath,amssymb,amsthm}  

\allowdisplaybreaks

\begin{document}

\def\fl#1{\left\lfloor#1\right\rfloor}
\def\cl#1{\left\lceil#1\right\rceil}
\def\ang#1{\left\langle#1\right\rangle}
\def\stf#1#2{\left[#1\atop#2\right]} 
\def\sts#1#2{\left\{#1\atop#2\right\}}
\def\eul#1#2{\left\langle#1\atop#2\right\rangle}
\def\N{\mathbb N}
\def\Z{\mathbb Z}
\def\R{\mathbb R}
\def\C{\mathbb C}

\newtheorem{theorem}{Theorem}
\newtheorem{Prop}{Proposition}
\newtheorem{Cor}{Corollary}
\newtheorem{Lem}{Lemma}
          
\newenvironment{Rem}{\begin{trivlist} \item[\hskip \labelsep{\it
Remark.}]\setlength{\parindent}{0pt}}{\end{trivlist}}

\title{Sylvester sums on the Frobenius set in arithmetic progression
}

\author{
Takao Komatsu
\\
\small Department of Mathematical Sciences, School of Science\\[-0.8ex]
\small Zhejiang Sci-Tech University\\[-0.8ex]
\small Hangzhou 310018 China\\[-0.8ex]
\small \texttt{komatsu@zstu.edu.cn}
}

\date{
\small MR Subject Classifications: Primary 11D07; Secondary 05A15, 05A17, 05A19, 11B68, 11D04, 11P81 
}

\maketitle 
 
\begin{abstract} 
Let $a_1,a_2,\dots,a_k$ be positive integers with $\gcd(a_1,a_2,\dots,a_k)=1$. 
The concept of the weighted sum $\sum_{n\in{\rm NR}}\lambda^{n}$ is introduced in \cite{KZ0,KZ}, where ${\rm NR}={\rm NR}(a_1,a_2,\dots,a_k)$ denotes the set of positive integers nonrepresentable in terms of $a_1,a_2,\dots,a_k$. When $\lambda=1$, such a sum is often called Sylvester sum. The main purpose of this paper is to give explicit expressions of the Sylvester sum ($\lambda=1$) and the weighed sum ($\lambda\ne 1$), where $a_1,a_2,\dots,a_k$ forms arithmetic progressions. As applications, various other cases are also considered, including weighted sums, almost arithmetic sequences, arithmetic sequences with an additional term, and geometric-like sequences. Several examples illustrate and confirm our results.   
\end{abstract}

\section{Introduction}  

Given positive integers $a_1,\dots,a_k$ with $\gcd(a_1,\dots,a_k)=1$, it is well-known that all sufficiently large $n$ can be represented as a nonnegative integer combination of $a_1,\dots,a_k$. 
The {\it Frobenius Problem} is to determine the largest positive integer that is NOT representable as a nonnegative integer combination of given positive integers that are coprime (see \cite{ra05} for general references). This number is denoted by $g(a_1,\dots,a_k)$ and often called Frobenius number. The problem to determine the Frobenius number has been often known as {\it Coin Exchange Problem} (or Postage Stamp Problem / Chicken McNugget Problem)  having a long history and is one of the problems that has attracted many people as well as experts. 

Let $n(a_1,\dots,a_k)$ be the number of positive integers with no nonnegative integer representation by $a_1,\dots,a_k$. It is sometimes called Sylvester number.  

According to Sylvester, for positive integers $a$ and $b$ with $\gcd(a,b)=1$,  
\begin{align*}
g(a,b)&=(a-1)(b-1)-1\quad{\rm \cite{sy1884}}\,,\\
n(a,b)&=\frac{1}{2}(a-1)(b-1)\quad{\rm \cite{sy1882}}\,. 
\end{align*}

There are many kinds of problems related to the Frobenius problem. The problems for the number of solutions (e.g., \cite{tr00}), and the sum of integer powers of values the gaps in numerical semigroups (e.g., \cite{bs93,fks,fr07}) are popular. In \cite{moree14}, the various results within the cyclotomic polynomial and numerical semigroup communities are better unified.  
One of other famous problems is about the so-called {\it Sylvester sums} 
$$
s(a_1,\dots,a_k):=\sum_{n\in{\rm NR}(a_1,\dots,a_k)}n 
$$ 
(see, e.g., \cite[\S 5.5]{ra05}, \cite{tu06} and references therein), where ${\rm NR}(a_1,\dots,a_k)$ denotes the set of positive integers without nonnegative integer representation by $a_1,\dots,a_k$. In addition, denote the set of positive integers with nonnegative integer representation by $a_1,\dots,a_k$ by ${\rm R}(a_1,\dots,a_k)$. 
For example, 
\begin{align*}
{\rm R}(4,7,11)&=\{0,4,7,8,11,12,14,15,16,18,19,20,21,\dots\}\quad{\rm (infinite)}\,,\\
{\rm NR}(4,7,11)&=\{1,2,3,5,6,9,10,13,17\}\quad{\rm (finite)}\,,
\end{align*} 
so $g(4,7,11)=17$.  
Brown and Shiue \cite{bs93} found the exact value for positive integers $a$ and $b$ with $\gcd(a,b)=1$,  
\begin{equation}
s(a,b)=\frac{1}{12}(a-1)(b-1)(2 a b-a-b-1)\,. 
\label{brown}
\end{equation} 
R\o dseth \cite{ro94} generalized Brown and Shiue's result by giving a closed form for $
s_\mu(a,b):=\sum_{n\in{\rm NR}(a,b)}n^\mu$,  
where $\mu$ is a positive integer.  

When $k=2$, there exist beautiful closed forms for Frobenius numbers, Sylvester numbers and Sylvester sums, but 
when $k\ge 3$, exact determination of these numbers is difficult.  
The Frobenius number cannot be given by closed formulas of a certain type (Curtis (1990) \cite{cu90}), the problem to determine $g(a_1,\dots,a_k)$ is NP-hard under Turing reduction (see, e.g., Ram\'irez Alfons\'in \cite{ra05}). Nevertheless, the Frobenius number for some special cases are calculated (e.g., \cite{op08,ro56,se77}). One convenient formula is by Johnson \cite{jo60}. One analytic approach to the Frobenius number can be seen in \cite{bgk01,ko03}.  
We consider a kind of generalizations called {\it weighted sum} for sum of numbers, which can be applied for the case of three or more variables. Notice that the case of two variables is given in \cite{KZ0}.    

Though closed forms for general case are hopeless for $k\ge 3$, several formulae for Frobenius numbers, Sylvester numbers and Sylvester sums have been considered under special cases.  For example, some formulas for the Frobenius number in three variables can be seen in \cite{tr17}. 

In fact, by introducing the other numbers, it is possible to determine the functions $g(A)$, $n(A)$ and $s(A)$ for the set of positive integers $A:=\{a_1,a_2,\dots,a_k\}$ with $\gcd(a_1,a_2,\dots,a_k)=1$. 

For each integer $i$ with $1\le i\le a_1-1$, there exists a least positive integer $m_i\equiv i\pmod{a_1}$ with $m_i\in{\rm R}(a_1,a_2,\dots,a_k)$. For convenience, we set $m_0=0$.  With the aid of such a congruence consideration modulo $a_1$, very useful results are established.  

\begin{Lem}  
We have 
\begin{align*}
g(a_1,a_2,\dots,a_k)&=\left(\max_{1\le i\le a_1-1}m_i\right)-a_1\,,\quad{\rm \cite{bs62}}\\ 
n(a_1,a_2,\dots,a_k)&=\frac{1}{a_1}\sum_{i=1}^{a_1-1}m_i-\frac{a_1-1}{2}\,,\quad{\rm \cite{se77}}\\ 
s(a_1,a_2,\dots,a_k)&=\frac{1}{2 a_1}\sum_{i=1}^{a_1-1}m_i^2-\frac{1}{2}\sum_{i=1}^{a_1-1}m_i+\frac{a_1^2-1}{12}\,.\quad{\rm \cite{tr08}}
\end{align*}
\label{lem1} 
\end{Lem} 
Note that the third formula appeared with a typo in \cite{tr08}, and it has been corrected in \cite{pu18,tr17b}. 
\bigskip 

In this paper, we treat with more general sums called {\it weighted sums}, defined by   
$$
s^{(\lambda)}(a_1,a_2,\dots,a_k):=\sum_{n\in{\rm NR}(a_1,a_2,\dots,a_k)}\lambda^n n\,. 
$$ 
This may be called {\it Sylvester weighted sums}.   
When $\lambda=1$, $s(a_1,a_2,\dots,a_k)=s^{(1)}(a_1,a_2,\dots,a_k)$ is the usual sum, though the obtained formulas are not included in the case of $\lambda\ne 1$.  When $\lambda\ne 1$, $s^{(-1)}(a_1,a_2,\dots,a_k)$ is the so-called alternate sum, which has been studied in \cite{tu06,wang08}.  

When the number of variables is two, similarly to the case of Frobenius number, Sylvester number and Sylvester sums, the results for weighted sums may be explicitly given. However, the results become complicated when the number of variables is bigger than or equal to three.  Nevertheless, if the sequence $a_1,a_2,\dots,a_k$ has some good regularities, the results are possible to be expressed explicitly.    
In this paper, we consider the weighted sum and (simple) sum of nonrepresentable numbers where $a_1=a$, $a_2=a+d$, $\dots$, $a_k=a+(k-1)d$ with $d>0$, $\gcd(a,d)=1$ and $k\le a_1$.  Some more varieties case are also given, including almost arithmetic sequences, arithmetic sequences with an additional term, and geometric-like sequences.

\section{Simple sum}
 
Let us begin from the simple sums.  

Let $a$, $d$ and $k$ be positive integers with $\gcd(a,d)=1$ and $2\le k\le a$. 
Roberts \cite{ro56} found the Frobenius number for arithmetic sequences. 
$$
g(a,a+d,\dots,a+(k-1)d)=\fl{\frac{a-2}{k-1}}a+(a-1)d\,.
$$ 
The case $d=1$ had been found by Brauer \cite{br42}.  
Selmer \cite{se77} generalized Roberts' result by giving a formula for almost arithmetic sequences. For a positive integer $h$, 
$$
g(a,h a+d,\dots,h a+(k-1)d)=\left(h\fl{\frac{a-2}{k-1}}+h-1\right)a+(a-1)d\,.
$$ 

Let $a-1=q(k-1)+r$ with $0\le r<k-1$. Grant \cite{gr73} obtained a formula for the number of positive integers with no nonnegative integer representation by arithmetic sequences.  
$$
n(a,a+d,\dots,a+(k-1)d)=\frac{1}{2}\bigl((a-1)(q+d)+r(q+1)\bigr)\,.
$$ 
Selmer \cite{se77} generalized Grant's result by giving a formula for almost arithmetic sequences. For a positive integer $h$, 
$$
n(a,h a+d,\dots,h a+(k-1)d)=\frac{1}{2}\bigl((a-1)(h q+d+h-1)+h r(q+1)\bigr)\,.
$$ 
Note that $q>0$ because $k\le a$. 
The sum of nonrepresentable numbers in arithmetic progression are given explicitly as follows.

\begin{theorem}  
Let $a$, $d$ and $k$ be positive integers with $\gcd(a,d)=1$ and $2\le k\le a$. Let $q$ and $r$ be nonnegative integers with $a-1=q(k-1)+r$ and $0\le r<k-1$. Then, 
\begin{align*}
&s(a,a+d,\dots,a+(k-1)d)\\
&=\frac{1}{12 q}\biggl(2 a q^3(a+2 r-1)+q^2\bigl(a d(4 a+4 r-5)-d(2 r-1)(r+1)+6 a r\bigr)\\
&\quad -q\bigl(3 d r(r+1)-(a-1)((a-1)(d^2-1)+a d^2)-2 a r(3 d+1)\bigr)\\
&\quad -d(a-r-1)^2\biggr)\,.
\end{align*}
\label{th1}
\end{theorem}

Substituting (\ref{lem5-1}) and (\ref{lem5-2}) in Lemma \ref{lem5} below into the third formula in Lemma \ref{lem1}, we can get Theorem \ref{th1}.   

\begin{Lem}  
When $a_1=a$, $a_2=a+d$, $\dots$, $a_k=a+(k-1)d$ with $d>0$, $\gcd(a,d)=1$ and $k\le a_1$, we have 
\begin{align}
&\sum_{i=1}^{a-1}m_i=\frac{a}{2}\bigl((a-1)(q+d+1)+r(q+1)\bigr)\,,
\label{lem5-1}\\ 
&\sum_{i=1}^{a-1}m_i^2=\frac{(q+1)\bigl((2 q+1)(a-r-1)+6 r(q+1)\bigr)}{6}a^2
+\frac{(a-1)a(2 a-1)}{6}d^2\notag\\
&\quad+2 a d(q+1)\biggl(\frac{(a-r-1)\bigl(q(4 a-4 r-1)-(a-r-1)\bigr)}{12 q}\notag\\
&\qquad +\frac{r(2 a-r-1)}{2}\biggr)\,.
\label{lem5-2}
\end{align}
\label{lem5}
\end{Lem}
\begin{proof}  
Since the minimal residue system $\{m_i\}$ ($1\le i\le a_1-1$) is given by 
\begin{align}
&a_2&\,&a_3&\,&\dots\dots&&a_{k-1}&\,&a_k\notag\\
&a_2+a_k&\,&a_3+a_k&\,&\dots\dots&&a_{k-1}+a_k&\,&2 a_k\notag\\
&\dots&\,&\dots&\,&\dots\dots&&\dots&\,&\dots\notag\\
&a_2+(q-1)a_k&\,&a_3+(q-1)a_k&\,&\dots\dots&&a_{k-1}+(q-1)a_k&\,&q a_k\notag\\
&a_2+q a_k&\,&a_3+q a_k&\,&\dots&&\!\!\!\!\!\!\!\!\!\!\!\! a_{r+1}+q a_k&&\,&
\label{minres}
\end{align}
(\cite[(3.8)]{se77}), the summation of all the elements is  
\begin{align} 
\sum_{i=1}^{a-1}m_i&=\bigl(1+2+\cdots+(q(k-1)+r)\bigr)d\notag\\
&\quad +\bigl((k-1)(1+2+\cdots+q)+r(q+1)\bigr)a\notag\\
&=\frac{(q+1)\bigl(q(k-1)+2 r\bigr)}{2}a+\frac{\bigl(q(k-1)+r\bigr)\bigl(q(k-1)+r+1\bigr)}{2}d\,. 
\label{eq:303}
\end{align}
Since $q(k-1)+r=a-1$, we have  
\begin{align*} 
\sum_{i=1}^{a-1}m_i&=\frac{(q+1)(a-1+r)}{2}a+\frac{(a-1)a}{2}d\\
&=\frac{a}{2}\bigl((a-1)(q+d+1)+r(q+1)\bigr)\,. 
\end{align*}
Similarly, in order to obtain (\ref{lem5-2}), we sum up all the elements 
\begin{align*}
&a_2^2&\,&a_3^2&\,&\dots\dots&&a_{k-1}^2&\,&a_k^2\\
&(a_2+a_k)^2&\,&(a_3+a_k)^2&\,&\dots\dots&&(a_{k-1}+a_k)^2&\,&(2 a_k)^2\\
&\dots&\,&\dots&\,&\dots\dots&&\dots&\,&\dots\\
&\bigl(a_2+(q-1)a_k\bigr)^2&\,&\bigl(a_3+(q-1)a_k\bigr)^2&\,&\dots\dots&&\bigl(a_{k-1}+(q-1)a_k\bigr)^2&\,&(q a_k)^2\\
&(a_2+q a_k)^2&\,&(a_3+q a_k)^2&\,&\dots&&\!\!\!\!\!\!\!\!\!\!\!\!(a_{r+1}+q a_k)^2&&\,&
\end{align*}
Then, we have 
\begin{align} 
\sum_{i=1}^{a-1}m_i^2&=\bigl(1^2+2^2+\cdots+(q(k-1)+r)^2\bigr)d^2\notag\\
&\quad +\bigl((k-1)(1^2+2^2+\cdots+q^2)+r(q+1)^2\bigr)a^2\notag\\
&\quad +2 a d\biggl(1+2+\cdots+(k-1)+2\bigl(k+(k+1)+\cdots+(2 k-2)\bigr)\notag\\
&\qquad +3\bigl((2 k-1)+(2 k)+\cdots+(3 k-3)\bigr)+\cdots\\
&\qquad +q\bigl(((q-1)k-q+2)+\cdots+(q k-q)\bigr)\notag\\
&\qquad +(q+1)\bigl((q k-q+1)+\cdots+(q k-q+r)\bigr)\biggr)\notag\\
&=\frac{\bigl(q(k-1)+r\bigr)\bigl(q(k-1)+r+1\bigr)\bigl(2 q(k-1)+2 r+1\bigr)}{6}d^2\notag\\
&\quad +\left(\frac{(k-1)q(q+1)(2 q+1)}{6}+r(q+1)^2\right)a^2\notag\\
&\quad +2 a d\biggl(\frac{k(k-1)q(q+1)}{4}+\frac{(k-1)^2(q-1)q(q+1)}{3}\notag\\
&\qquad +(q+1)\biggl((q k-q)r+\frac{r(r+1)}{2}\biggr)\biggr)\,. 
\label{eq:304}
\end{align}
Since $q(k-1)+r=a-1$, we have  
\begin{align*} 
\sum_{i=1}^{a-1}m_i^2&=\frac{(q+1)\bigl((2 q+1)(a-r-1)+6 r(q+1)\bigr)}{6}a^2
+\frac{(a-1)a(2 a-1)}{6}d^2\\
&\quad+2 a d(q+1)\biggl(\frac{(a-r-1)\bigl(q(4 a-4 r-1)-(a-r-1)\bigr)}{12 q}\\
&\qquad +\frac{r(2 a-r-1)}{2}\biggr)\,. 
\end{align*}
\end{proof}

\subsection{Examples}  

When $a=7$, $r=2$ and $k=3$, we get $q=3$ and $r=0$. Hence, by Theorem \ref{th1}, we have $s(7,9,11,13)=165$. In fact, the sum of nonrepresentable numbers is 
$$
1 + 2 + 3 + 4 + 5 + 6 + 8 + 10 + 12 + 13 + 15 + 17 + 19 + 24 + 26=165\,. 
$$
When $a=7$, $r=3$ and $k=3$, we get $q=3$ and $r=0$. Hence, $s(7,10,13,16)=237$. The sum of nonrepresentable numbers is 
$$
1 + 2 + 3 + 4 + 5 + 6 + 8 + 9 + 11 + 12 + 15 + 16 + 18 + 19 + 22 + 25 + 29 + 32=237\,. 
$$
When $a=6$, $r=5$ and $k=4$, we get $q=1$ and $r=2$. Hence, $s(6,11,16,21)=212$. The sum of nonrepresentable numbers is 
$$
1 + 2 + 3 + 4 + 5 + 7 + 8 + 9 + 10 + 13 + 14 + 15 + 19 + 20 + 25 + 26 + 31=212\,.
$$

\subsection{Almost arithmetic progressions}   

Theorem \ref{th1} can be extended to that of almost arithmetic progressions.   Nonnegative integers $q$ and $r$ are similarly determined by $a-1=q(k-1)+r$ with $0\le r<k-1$.  

\begin{theorem}  
For $a,d,h>0$ with $\gcd(a,d)=1$ and $k\le a$, we have 
\begin{align*}
&s(a,h a+d,\dots,h a+(k-1)d)\\
&=\frac{1}{12 q}\biggl(2 h^2 a q^3(a+2 r-1)\\
&\quad +q^2\bigl(3 h^2 a(a+3 r-1)+h(a d(4 a+4 r-5)-d(2 r-1)(r+1)-3 a(a+r-1)\bigr)\\
&\quad +q\bigl(h^2 a(a+5 r-1)+(a-1)(d-1)(2 a d-a-d-1)\\
&\qquad +3 h((a+r-1)(d(a+r)-a)-2 d r^2)\bigr)\\
&\quad -h d(a-r-1)^2\biggr)\,.
\end{align*}
\label{th1a}  
\end{theorem}  

This result is similarly proved from the following sums, which are analogous ones of those in  Lemma \ref{lem5}.   

\begin{Lem}  
When $a_1=a$, $a_2=h a+d$, $\dots$, $a_k=h a+(k-1)d$ with $d,h>0$, $\gcd(a,d)=1$ and $k\le a_1$, we have 
\begin{align*}
&\sum_{i=1}^{a-1}m_i=\frac{a}{2}\bigl((a-1)(h(q+1)+d)+h r(q+1)\bigr)\,,
\\ 
&\sum_{i=1}^{a-1}m_i^2=\frac{(q+1)\bigl((2 q+1)(a-r-1)+6 r(q+1)\bigr)}{6}h^2 a^2
+\frac{(a-1)a(2 a-1)}{6}d^2\\
&\quad+2 h a d(q+1)\biggl(\frac{(a-r-1)\bigl(q(4 a-4 r-1)-(a-r-1)\bigr)}{12 q} 
+\frac{r(2 a-r-1)}{2}\biggr)\,.
\end{align*}
\label{lem5a}
\end{Lem}

\section{Weighted sums} 

In this section, we give a formula for the weighted sums 
$$
s^{(\lambda)}(a,a+d,\dots,a+(k-1)d):=\sum_{n\in{\rm NR}(a,a+d,\dots,a+(k-1)d)}\lambda^n n\,.  
$$ 
As in the previous sections, let $a_1=a$, $a_2=a+d$, $\dots$, $a_k=a+(k-1)d$ with $d>0$, $\gcd(a,d)=1$ and $k\le a_1$.  Let $a-1=q(k-1)+r$ with $0\le r<k-1$. 

\begin{theorem}  
For $\lambda\ne 0$ with $\lambda^d\ne 1$, $\lambda^a\ne 1$ and $\lambda^{a_k}\ne 1$, we have 
\begin{align*}
&s^{(\lambda)}(a,a+d,\dots,a+(k-1)d)\\
&=\frac{1}{\lambda^{a}-1}\biggl(\frac{\lambda^d(\lambda^{q a_k}-\lambda^{a_k})}{\lambda^{a_k}-1}\left(\frac{a_k\lambda^{a_k}-a\lambda^a}{\lambda^d-1}-\frac{d(\lambda^{a_k}-\lambda^a)}{(\lambda^d-1)^2}\right)\\
&\qquad +\frac{\lambda^d(\lambda^{a_k}-\lambda^a)}{\lambda^d-1}\left(\frac{q a_k\lambda^{q a_k}}{\lambda^{a_k}-1}-\frac{a_k\lambda^{a_k}(\lambda^{q a_k}-1)}{(\lambda^{a_k}-1)^2}\right)\\ 
&\qquad +\lambda^{q a_k+d}\left(\frac{a_{r+1}\lambda^{a_{r+1}}-a\lambda^a}{\lambda^d-1}-\frac{d(\lambda^{a_{r+1}}-\lambda^a)}{(\lambda^d-1)^2}\right)\\
&\qquad +q a_k\lambda^{q a_k+d}\frac{\lambda^{a_{r+1}}-\lambda^a}{\lambda^d-1}\biggr)\\ 
&\quad -\frac{a\lambda^{a}}{(\lambda^{a}-1)^2}\biggl(1+\frac{\lambda^d(\lambda^{a_k}-\lambda^a)}{\lambda^d-1}\frac{\lambda^{q a_k}-1}{\lambda^{a_k}-1} 
+\frac{\lambda^{q a_k+d}(\lambda^{a_{r+1}}-\lambda^a)}{\lambda^d-1}\biggr)\\
&\quad +\frac{\lambda}{(\lambda-1)^2}\,.  
\end{align*}
\label{th10}
\end{theorem}

The proof is based upon the following results.  

\begin{Lem}  
\begin{align}
\sum_{i=1}^{a_1-1}m_i\lambda^{m_i}&=\frac{\lambda^d(\lambda^{q a_k}-\lambda^{a_k})}{\lambda^{a_k}-1}\left(\frac{a_k\lambda^{a_k}-a\lambda^a}{\lambda^d-1}-\frac{d(\lambda^{a_k}-\lambda^a)}{(\lambda^d-1)^2}\right)\notag\\
&\quad +\frac{\lambda^d(\lambda^{a_k}-\lambda^a)}{\lambda^d-1}\left(\frac{q a_k\lambda^{q a_k}}{\lambda^{a_k}-1}-\frac{a_k\lambda^{a_k}(\lambda^{q a_k}-1)}{(\lambda^{a_k}-1)^2}\right)\notag\\ 
&\quad +\lambda^{q a_k+d}\left(\frac{a_{r+1}\lambda^{a_{r+1}}-a\lambda^a}{\lambda^d-1}-\frac{d(\lambda^{a_{r+1}}-\lambda^a)}{(\lambda^d-1)^2}\right)\notag\\
&\quad +q a_k\lambda^{q a_k+d}\frac{\lambda^{a_{r+1}}-\lambda^a}{\lambda^d-1}\,,
\label{10-1}\\
\sum_{i=1}^{a_1-1}\lambda^{m_i}&=\frac{\lambda^d(\lambda^{a_k}-\lambda^a)}{\lambda^d-1}\frac{\lambda^{q a_k}-1}{\lambda^{a_k}-1}
+\frac{\lambda^{q a_k+d}(\lambda^{a_{r+1}}-\lambda^a)}{\lambda^d-1}\,. 
\label{10-0}
\end{align}
\label{lem10} 
\end{Lem}  

\noindent 
{\it Remark.}  
The last two terms in (\ref{10-1}) and the last term in (\ref{10-0}) are equal to $0$ when $r=0$.  
\bigskip 

\begin{proof}[Proof of Lemma \ref{lem10}.]
(\ref{10-1}) is obtained by summing up all the elements 
{\small 
\begin{align*}
&a_2\lambda^{a_2}&\,&\dots\dots&&a_{k-1}\lambda^{a_{k-1}}&\,&a_k\lambda^{a_k}\\
&(a_2+a_k)\lambda^{a_2+a_k}&\,&\dots\dots&&(a_{k-1}+a_k)\lambda^{a_{k-1}+a_k}&\,&(2 a_k)\lambda^{2 a_k}\\
&\dots&\,&\dots\dots&&\dots&\,&\dots\\
&\bigl(a_2+(q-1)a_k\bigr)\lambda^{a_2+(q-1)a_k}&\,&\dots\dots&&\bigl(a_{k-1}+(q-1)a_k\bigr)\lambda^{a_{k-1}+(q-1)a_k}&\,&(q a_k)\lambda^{q a_k}\\
&(a_2+q a_k)\lambda^{a_2+q a_k}&\,&\dots&&\!\!\!\!\!\!\!\!\!\!\!\!(a_{r+1}+q a_k)\lambda^{a_{r+1}+q a_k}&&\,&
\end{align*}
}  
Similarly, (\ref{10-0}) is obtained by summing up all the elements 
\begin{align*}
&\lambda^{a_2}&\,&\dots\dots&&\lambda^{a_{k-1}}&\,&\lambda^{a_k}\\
&\lambda^{a_2+a_k}&\,&\dots\dots&&\lambda^{a_{k-1}+a_k}&\,&\lambda^{2 a_k}\\
&\dots&\,&\dots\dots&&\dots&\,&\dots\\
&\lambda^{a_2+(q-1)a_k}&\,&\dots\dots&&\lambda^{a_{k-1}+(q-1)a_k}&\,&\lambda^{q a_k}\\
&\lambda^{a_2+q a_k}&\,&\dots&&\!\!\!\!\!\!\!\!\!\!\!\!\lambda^{a_{r+1}+q a_k}&&\,&
\end{align*}
\end{proof}

We also need the following formula in \cite{KZ}.  

\begin{Lem}  
If $\lambda\ne 0$ and $\lambda^{a_1}\ne1$, then  
\begin{align*}  
&s^{(\lambda)}(a_1,a_2,\dots,a_k)\\
&=\frac{1}{\lambda^{a_1}-1}\sum_{i=0}^{a_1-1}m_i\lambda^{m_i}
-\frac{a_1\lambda^{a_1}}{(\lambda^{a_1}-1)^2}\sum_{i=0}^{a_1-1}\lambda^{m_i}+\frac{\lambda}{(\lambda-1)^2}\,.
\end{align*}  
\label{lem11}
\end{Lem}

\begin{proof}[Proof of Theorem \ref{th10}.] 
Substituting (\ref{10-1}) and (\ref{10-0}) in Lemma \ref{lem10} into the formula in Lemma \ref{lem11}, we get the desired result. Notice that we do need an additional quantity $\lambda^{m_0}=\lambda^0=1$ in the second term in Lemma \ref{lem11}.  
\end{proof}

\subsection{Examples}  
When $a=7$, $r=2$ and $k=3$, we get $q=3$ and $r=0$. If $\lambda=2$ and $d=2$, then by Theorem \ref{th10},
we have $s^{(2)}(7, 9, 11)=2160333442$. In fact, the sum of nonrepresentable numbers
is 
\begin{align*}
&2\cdot 1 + 2^2\cdot 2 + 2^3\cdot 3 + 2^4\cdot 4 + 2^5\cdot 5 + 2^6\cdot 6 + 2^8\cdot 8 + 2^{10}\cdot 10\\ 
&\quad + 2^{12}\cdot 12 + 2^{13}\cdot 13 + 2^{15}\cdot 15 + 2^{17}\cdot 17 + 2^{19}\cdot 19 + 2^{24}\cdot 24 + 2^{26}\cdot 26\\ 
&=2160333442\,. 
\end{align*}

When $a=6$, $r=5$ and $k=4$, we get $q=1$ and $r=2$. If $\lambda=\sqrt{-1}$ and $d=5$, then $s^{(\sqrt{-1})}(6, 11, 16, 21)=-20-22\sqrt{-1}$.

\section{Arithmetic sequences with an additional term}  

Consider the case 
$$
a_1=a,\, a_2=a+d,\, a_3=a+2 d,\, \dots,\, a_k=a+(k-1)d,\, a_{k+1}=a+K d\,,  
$$   
where $\gcd(a,d)=1$, $K>k$ and $a\ge k$. Put 
\begin{align}  
K-1&=q(k-1)+r,\quad 0\le r<k-1\,,\notag\\  
a&=\alpha K+\beta, \quad 0\le\beta<K\,.
\label{eq:aa} 
\end{align}
In \cite[(3.16)]{se77}, as $d=1$, it is shown that 
\begin{align*}
&n(a,a+1,a+2,\dots,a+k-1,a+K)\\
&=\frac{\alpha\bigl(a+(q+1)(r-1)+q K+\beta+\bigr)}{2}+\frac{(\gamma+1)(\beta+\delta-1)}{2}\,.
\end{align*}
Some more special cases of the number of nonrepresentable numbers are given in \cite{se77}.

The minimal residue system $1,2,\dots,a-1\pmod a$, where all residues appear once and only once, can be constructed as follows. The case $d=1$ is illustrated in (\cite{se77}), but we explain here again as the general $d$ case.   

The first line of this minimal residue system is the same as the whole numbers in (\ref{minres}). There are totally $(k-1)q+r=K-1$ elements, which consist the residue system $\{d,2 d,\dots,(K-1)d\}\pmod{a}$. The second line is of $a_{k+1}$, and $a_{k+1}$ plus each number in (\ref{minres}), representing the residue system  $\{K d,(K+1)d,\dots,(2 K-1)d\}\pmod{a}$. Hence, there are totally $K$ elements in the second line. Similarly, the $j$-th line ($1\le j\le\alpha-1$) is of $(j-1)a_{k+1}$, and $(j-1)a_{k+1}$ plus each number in (\ref{minres}), representing the residue system  $\{(j-1)K d,((j-1)K+1)d,\dots,(j K-1)d\}\pmod{a}$. Hence, there are totally $K$ elements in the $j$-th line ($2\le j\le\alpha-1$). 
The $\alpha$-th line ends with the element 
$$
t'=\begin{cases}
a_{r+1}+q a_k+(\alpha-1)a_{k+1}=(q+\alpha+1)a-(\beta+1)d&\text{if $r>0$};\\
q a_k+(\alpha-1)a_{k+1}=(q+\alpha)a-(\beta+1)d&\text{if $r=0$}\,. 
\end{cases}
$$ 
If $\beta=0$, then we have already gotten the minimal residue system. Otherwise, put  
\begin{equation}
\beta-1=\gamma(k-1)+\delta,\quad 0\le\delta<k-1\,. 
\label{beta-gamma} 
\end{equation}
The final line is of $\beta$ elements and ends with the element 
$$
t''=\begin{cases}
a_{\delta+1}+\gamma a_k+\alpha a_{k+1}=(\gamma+\alpha+2)a-d&\text{if $\delta>0$};\\
\gamma a_k+\alpha a_{k+1}=(\gamma+\alpha+1)a-d&\text{if $\delta=0$}\,. 
\end{cases}
$$

In (\ref{eq:303}), set $q(k-1)+r=K-1$ instead of $q(k-1)+r=a-1$. Then we have 
\begin{align*}
S_1&=S_1(a,d,q,r)\\
&:=\frac{(q+1)(K-1+r)}{2}a+\frac{(K-1)K}{2}d\,. 
\end{align*} 
In (\ref{eq:304}), set $q(k-1)+r=K-1$ instead of $q(k-1)+r=a-1$. Then we have 
\begin{align*} 
S_2&=S_2(a,d,q,r)\\
&:=\frac{(q+1)\bigl((2 q+1)(K-r-1)+6 r(q+1)\bigr)}{6}a^2
+\frac{(K-1)K(2 K-1)}{6}d^2\\
&\quad+2 a d(q+1)\biggl(\frac{(K-r-1)\bigl(q(4 K-4 r-1)-(K-r-1)\bigr)}{12 q}\\
&\qquad +\frac{r(2 K-r-1)}{2}\biggr)\,. 
\end{align*} 
When $\beta=0$, the whole summation of the least elements modulo $i\pmod a$ ($1\le i\le a-1$) is equal to 
\begin{align*}
\sum_{i=1}^{a-1}m_i&=\sum_{j=1}^\alpha\bigl((j-1)a_{k+1}K+S_1\bigr)\\
&=\frac{\alpha(\alpha-1)K(a+K d)}{2}+\alpha S_1
\,. 
\end{align*}
(When $\beta>0$, we need more additional elements from the ($\alpha+1$)-st line, whose sum is denoted by $T_1$.) The whole square summation is equal to 
\begin{align*} 
\sum_{i=1}^{a-1}m_i^2&=
\sum_{j=1}^\alpha\bigl((j-1)^2 a_{k+1}^2 K+2(j-1)a_{k+1}S_1+S_2\bigr)\\ 
&=\frac{\alpha(\alpha-1)(2\alpha-1)(a+K d)^2 K}{6}+\alpha(\alpha-1)(a+K d)S_1+\alpha S_2
\,. 
\end{align*} 
(When $\beta>0$, we need more additional elements from the ($\alpha+1$)-st line, whose sum is denoted by $T_2$.) 
Substituting them into the third formula in Lemma \ref{lem1}, by $\alpha=a K$, we get 
\begin{align*}  
&s(a,a+d,a+2 d,\dots,a+(k-1)d,a+K d)\\ 
&=\frac{1}{2 a}\left(\frac{\alpha(\alpha-1)(2\alpha-1)(a+K d)^2 K}{6}+\alpha(\alpha-1)(a+K d)S_1+\alpha S_2\right)\\ 
&\quad -\frac{1}{2}\left(\frac{\alpha(\alpha-1)K(a+K d)}{2}+\alpha S_1\right)
+\frac{a^2-1}{12}\\
&=\frac{\alpha(\alpha-1)(a+K d)K\bigl(2(\alpha-2)a+(2\alpha-1)K d\bigr)}{12 a}\\
&\quad +\frac{\alpha\bigl((\alpha-2)a+(\alpha-1)K d\bigr)}{2 a}S_1
+\frac{\alpha S_2}{2 a}+\frac{a^2-1}{12}\\
&=\frac{(a-K)(a+K d)\bigl(2 a^2+2 a K(d-2)-K^2 d\bigr)}{12 K^2}+\frac{a^2+a K(d-2)-K^2 d}{2 K^2}S_1\\
&\quad +\frac{S_2}{2 K}+\frac{a^2-1}{12}\,. 
\end{align*} 

When $\beta>0$, from (\ref{beta-gamma}) the sum of additional terms 
\begin{align*}  
&\alpha a_{k+1}&&&&&&\\
&\alpha a_{k+1}+a_2&&\alpha a_{k+1}+a_3&&\dots&&\alpha a_{k+1}+a_k\\ 
&\alpha a_{k+1}+a_k+a_2&&\alpha a_{k+1}+a_k+a_3&&\dots&&\alpha a_{k+1}+2 a_k\\
&\dots&&\dots&&\dots&&\dots\\
&\alpha a_{k+1}+(\gamma-1)a_k+a_2&&\alpha a_{k+1}+(\gamma-1)a_k+a_3&&\dots&&\alpha a_{k+1}+\gamma a_k\\
&\alpha a_{k+1}+\gamma a_k+a_2&&\dots&&&&\hskip-3cm\!\!\!\!\!\!\!\!\!\alpha a_{k+1}+\gamma a_k+a_{\delta+1}\\
\end{align*}
is given by 
\begin{align*}  
&T_1=T_1(a,d,K,\alpha,\beta,\delta,\gamma)\\ 
&:=\alpha\beta a_{k+1}+\gamma(a_2+a_3+\cdots+a_k)\\  
&\quad +(k-1)a_k\bigl(1+2+\cdots+(\gamma-1)\bigr)+\gamma\delta a_k+(a_2+\cdots+a_{\delta+1})\\
&=\alpha\beta(a+K d)+\gamma\left((k-1)a+\frac{(k-1)k}{2}d\right)
+(k-1)\bigl(a+(k-1)d\bigr)\frac{\gamma(\gamma-1)}{2}\\
&\quad +\gamma\delta\bigl(a+(k-1)d\bigr)+\delta a+\frac{\delta(\delta+1)}{2}d\\
&=\left(\alpha(\beta-1)+\frac{(k-1)\gamma(\gamma+1)}{2}+\delta(\gamma+1)\right)a\\
&\quad+\left(\alpha(\beta-1)K+\frac{(k-1)\gamma(\gamma k-\gamma+1)}{2}+\frac{\delta\bigl(2\gamma(k-1)+\delta+1\bigr)}{2}\right)d\\
&=\biggl(\alpha\beta+\frac{(\beta+\delta-1)(\gamma+1)}{2}\biggr)a\\
&\quad+\biggl(\alpha\beta K+\frac{(\beta-\delta-1)(\beta-\delta)}{2}+\frac{\delta(2\beta-\delta-1)}{2}\biggr)d\,. 
\end{align*}
The sum of the square of additional terms is given by 
$$ 
T_2=T_2(a,d,K,\beta,\delta,\gamma):=\beta(\alpha a_{k+1})^2+2\alpha a_{k+1}(T_1-\alpha\beta a_{k+1})+T_3\,, 
$$
where 
\begin{align*}  
T_3&=\gamma(a_2^2+a_3^2+\cdots+a_k^2)\\
&\quad +2 a_k\bigl(1+2+\cdots+(\gamma-1)\bigr)(a_2+a_3+\cdots+a_k)\\
&\quad +(k-1)a_k^2\bigl(1^2+2^2+\cdots+(\gamma-1)^2\bigr)\\
&\quad +\delta\gamma^2 a_k^2+2\gamma a_k(a_2+a_3+\cdots+a_{\delta+1})\\
&\quad +(a_2^2+a_3^2+\cdots+a_{\delta+1}^2)\\
&=\gamma\left((k-1)a^2+(k-1)k a d+\frac{(k-1)k(2 k-1)}{6}d^2\right)\\
&\quad +\bigl(a+(k-1)d\bigr)(\gamma-1)\gamma\left((k-1)a+\frac{(k-1)k}{2}d\right)\\ 
&\quad +(k-1)\bigl(a+(k-1)d\bigr)^2\frac{(\gamma-1)\gamma(2\gamma-1)}{6}\\
&\quad +\delta\gamma^2\bigl(a+(k-1)d\bigr)^2+2\gamma\bigl(a+(k-1)d\bigr)\left(\delta a+\frac{\delta(\delta+1)}{2}d\right)\\
&\quad +\delta a^2+\delta(\delta+1)a d+\frac{\delta(\delta+1)(2\delta+1)}{6}d^2\\
&=\left(\frac{(\beta-\delta-1)(\gamma+1)(2\gamma+1)}{6}+\delta(\gamma+1)^2\right)a^2\\
&\quad +\left(\frac{(\beta-\delta-1)\bigl(4(\beta-\delta)-k\bigr)}{6}+\delta(2\beta-\delta-1)\right)(\gamma+1)a d\\
&\quad +\frac{(\beta-1)\beta(2\beta-1)}{6}d^2\,.   
\end{align*}

When $\beta>0$, the whole summation of the least elements modulo $i\pmod a$ ($1\le i\le a-1$) is equal to 
\begin{align*}
\sum_{i=1}^{a-1}m_i&=\sum_{j=1}^\alpha\bigl((j-1)a_{k+1}K+S_1\bigr)+T_1\\
&=\frac{\alpha(\alpha-1)K(a+K d)}{2}+\alpha S_1+T_1
\,. 
\end{align*}
The whole square summation is equal to 
\begin{align*} 
\sum_{i=1}^{a-1}m_i^2&=
\sum_{j=1}^\alpha\bigl((j-1)^2 a_{k+1}^2 K+2(j-1)a_{k+1}S_1+S_2\bigr)+T_2\\ 
&=\frac{\alpha(\alpha-1)(2\alpha-1)(a+K d)^2 K}{6}\\
&\quad +\alpha(\alpha-1)(a+K d)S_1+\alpha S_2+T_2
\,.
\end{align*} 
Substituting them into the third formula in Lemma \ref{lem1}, we get 
\begin{align*}  
&s(a,a+d,a+2 d,\dots,a+(k-1)d,a+K d)\\
&=\frac{1}{2 a}\biggl(\frac{\alpha(\alpha-1)(2\alpha-1)(a+K d)^2 K}{6}\\
&\quad +\alpha(\alpha-1)(a+K d)S_1+\alpha S_2+T_2\biggr)\\ 
&\quad -\frac{1}{2}\left(\frac{\alpha(\alpha-1)K(a+K d)}{2}+\alpha S_1+T_1\right)
+\frac{a^2-1}{12}\\
&=\frac{\alpha(\alpha-1)(a+K d)K\bigl(2(\alpha-2)a+(2\alpha-1)K d\bigr)}{12 a}\\
&\quad +\frac{\alpha\bigl((\alpha-2)a+(\alpha-1)K d\bigr)}{2 a}S_1\\
&\quad +\frac{\alpha S_2}{2 a}-\frac{T_1}{2}+\frac{T_2}{2 a}+\frac{a^2-1}{12}\,. 
\end{align*}

\begin{theorem}  
Assume that $\gcd(a,d)=1$, $K>k$ and $a\ge k$. Integers $q$, $r$, $\alpha$, $\beta$, $\gamma$ and $\delta$ are decided as in (\ref{eq:aa}) and (\ref{beta-gamma}).  
Then we have 
\begin{align*} 
&s(a,a+d,a+2 d,\dots,a+(k-1)d,a+K d)\\
&=\frac{\alpha(\alpha-1)(a+K d)K\bigl(2(\alpha-2)a+(2\alpha-1)K d\bigr)}{12 a}\\
&\quad +\frac{\alpha\bigl((\alpha-2)a+(\alpha-1)K d\bigr)}{2 a}S_1\\
&\quad +\frac{\alpha S_2}{2 a}-\frac{T_1}{2}+\frac{T_2}{2 a}+\frac{a^2-1}{12}\,, 
\end{align*} 
where    
\begin{align*}  
S_1&=\frac{(q+1)(K-1+r)}{2}a+\frac{(K-1)K}{2}d\,,\\ 
S_2&=\frac{(q+1)\bigl((2 q+1)(K-r-1)+6 r(q+1)\bigr)}{6}a^2
+\frac{(K-1)K(2 K-1)}{6}d^2\\
&\quad+2 a d(q+1)\biggl(\frac{(K-r-1)\bigl(q(4 K-4 r-1)-(K-r-1)\bigr)}{12 q}\\
&\qquad +\frac{r(2 K-r-1)}{2}\biggr)\,. 
\end{align*}
When $a\mid K$, $T_1=T_2=0$. When $a\nmid K$, 
\begin{align*} 
&T_1=T_1(a,d,K,\alpha,\beta,\delta,\gamma)\\ 
&=\biggl(\alpha\beta+\frac{(\beta+\delta-1)(\gamma+1)}{2}\biggr)a\\
&\quad+\biggl(\alpha\beta K+\frac{(\beta-\delta-1)(\beta-\delta)}{2}+\frac{\delta(2\beta-\delta-1)}{2}\biggr)d  
\end{align*}
and  
$$ 
T_2=T_2(a,d,K,\beta,\delta,\gamma):=\beta(\alpha a_{k+1})^2+2\alpha a_{k+1}(T_1-\alpha\beta a_{k+1})+T_3\,, 
$$
where 
\begin{align*}  
T_3&=\left(\frac{(\beta-\delta-1)(\gamma+1)(2\gamma+1)}{6}+\delta(\gamma+1)^2\right)a^2\\
&\quad +\left(\frac{(\beta-\delta-1)\bigl(4(\beta-\delta)-k\bigr)}{6}+\delta(2\beta-\delta-1)\right)(\gamma+1)a d\\
&\quad +\frac{(\beta-1)\beta(2\beta-1)}{6}d^2\,.   
\end{align*}
\label{th-aa}
\end{theorem}

\noindent 
{\it Remark.}  
When $K=a$ and $\beta=0$, by $\sum_{i=1}^{a-1}m_i=S_1$ and $\sum_{i=1}^{a-1}m_i^2=S_2$, Theorem \ref{th-aa} is reduced to Theorem \ref{th1}. 


If $\alpha\nmid K$, then by $a=\alpha K+\beta$, we have 
\begin{align*}  
&s(a,a+d,a+2 d,\dots,a+(k-1)d,a+K d)\\
&=\frac{(a-K)(a+K d)\bigl(2 a^2+2 a K(d-2)-K^2 d\bigr)}{12 K^2}+\frac{a^2+a K(d-2)-K^2 d}{2 K^2}S_1\\
&\quad +\frac{S_2}{2 K}+\frac{a^2-1}{12}\\
&=\frac{(a-\beta)(a-\beta-K)(a+K d)\bigl(2(a-\beta-2 K)a+(2 a-2\beta-K)d\bigr)}{12 a K^2}\\
&\quad +\frac{(a-\beta)\bigl((a-\beta-2 K)a+(a-\beta-K)d\bigr)}{2 a K^2}S_1\\ 
&\quad +\frac{a-\beta}{2 a K}S_2-\frac{T_1}{2}+\frac{T_2}{2 a}+\frac{a^2-1}{12}\,.
\end{align*}

\bigskip 

As a special case,  
\begin{align*}
g(a,a+1,a+2,a+4)&=(a+1)\fl{\frac{a}{4}}+\fl{\frac{a+1}{4}}+2\fl{\frac{a+2}{4}}-1\,,\quad{\rm \cite{dm64}}\\ 
n(a,a+1,a+2,a+4)&=\fl{\frac{a(a+4)}{8}}\quad{\rm \cite{se77}}
\end{align*} 
are found, where $\fl{x}$ denotes the integer part of a real $x$.  
We can give the correspondence summations. 

\begin{Cor}
\begin{align*} 
&s(a,a+1,a+2,a+4)\\
&=\begin{cases}
\frac{1}{96}(a^4+8 a^3+26 a^2+16 a)&\text{if $a\equiv 0\pmod 4$};\\ 
\frac{1}{96}(a^4+8 a^3+11 a^2-38 a+18)&\text{if $a\equiv 1\pmod 4$};\\ 
\frac{1}{96}(a^4+8 a^3+14 a^2-32 a+24)&\text{if $a\equiv 2\pmod 4$};\\ 
\frac{1}{96}(a^4+8 a^3+11 a^2-50 a+42)&\text{if $a\equiv 3\pmod 4$}\,.
\end{cases}
\end{align*}
\label{cor:a124}
\end{Cor}
\begin{proof} 
Here, $d=1$, $k=3$, $K=4$, $q=r=1$ and $\alpha=\fl{a/4}$. 
When $a\equiv 0\pmod 4$, $\beta=0$.     
When $a\equiv 1\pmod 4$, $\beta=1$, $\gamma=\delta=0$.  
When $a\equiv 2\pmod 4$, $\beta=2$, $\gamma=0$ and $\delta=1$.  
When $a\equiv 3\pmod 4$, $\beta=3$, $\gamma=1$ and $\delta=0$. 
The results follow from Theorem \ref{th-aa}. 
\end{proof} 

In \cite{dm64}, some more special cases are found:  
\begin{align*}  
&g(a,a+1,a+2,a+5)\\
&=a\fl{\frac{a+1}{5}}+\fl{\frac{a}{5}}+\fl{\frac{a+1}{5}}+\fl{\frac{a+2}{5}}+2\fl{\frac{a+3}{5}}-1\,,\\ 
&g(a,a+1,a+2,a+6)\\
&=a\fl{\frac{a}{6}}+2\fl{\frac{a}{6}}+2\fl{\frac{a+1}{6}}+5\fl{\frac{a+2}{6}}+\fl{\frac{a+3}{6}}+\fl{\frac{a+4}{6}}+\fl{\frac{a+5}{6}}-1\,. 
\end{align*}  
We can similarly derive the correspondence result by Theorem \ref{th-aa}. 

\begin{Cor}
\begin{align*} 
&s(a,a+1,a+2,a+5)\\
&=\begin{cases}
\frac{1}{150}(a^4+13 a^3+65 a^2+35 a)&\text{if $a\equiv 0\pmod 5$};\\ 
\frac{1}{150}(a^4+13 a^3+41 a^2-85 a+30)&\text{if $a\equiv 1\pmod 5$};\\ 
\frac{1}{150}(a^4+13 a^3+41 a^2-97 a+60)&\text{if $a\equiv 2\pmod 5$};\\ 
\frac{1}{150}(a^4+13 a^3+35 a^2-139 a+120)&\text{if $a\equiv 3\pmod 5$};\\ 
\frac{1}{150}(a^4+13 a^3+53 a^2-19 a+90)&\text{if $a\equiv 4\pmod 5$}\,.
\end{cases}
\end{align*}
\label{cor:a125}
\end{Cor}

\begin{Cor}
\begin{align*} 
&s(a,a+1,a+2,a+6)\\
&=\begin{cases}
\frac{1}{216}(a^4+21 a^3+189 a^2+126 a)&\text{if $a\equiv 0\pmod 6$};\\ 
\frac{1}{216}(a^4+21 a^3+150 a^2-169 a-3)&\text{if $a\equiv 1\pmod 6$};\\ 
\frac{1}{216}(a^4+21 a^3+141 a^2-290 a+48)&\text{if $a\equiv 2\pmod 6$};\\ 
\frac{1}{216}(a^4+21 a^3+126 a^2-441 a+189)&\text{if $a\equiv 3\pmod 6$};\\ 
\frac{1}{216}(a^4+21 a^3+141 a^2-322 a+240)&\text{if $a\equiv 4\pmod 6$};\\ 
\frac{1}{216}(a^4+21 a^3+150 a^2-281 a+237)&\text{if $a\equiv 5\pmod 6$}\,.
\end{cases}
\end{align*}
\label{cor:a126}
\end{Cor}

\subsection{Examples}  

Consider the sequence $12,17,22,27,42$. Then, $a=12$, $d=5$ $k=4$, $K=6$, $q=1$, $r=2$, $\alpha=2$ and $\beta=0$. By Theorem \ref{th-aa}, we have $s(12,17,22,27,42)=1211$. In fact, the sum of nonrepresentable numbers is 
\begin{align*} 
&1+2+3+4+5+6+7+8+9+10+11+13+14+15+16+18\\
&+19+20+21+23+25+26+28+30+31+32+33+35+37\\
&+38+40+43+45+47+50+52+55+57+62+67+74+79\\
&=1211\,. 
\end{align*} 

Consider the sequence $14,17,20,23,38$. Then, $a=14$, $d=3$ $k=4$, $K=8$, $q=2$, $r=1$, $\alpha=1$, $\beta=6$, $\gamma=1$ and $\delta=2$. By Theorem \ref{th-aa}, we have $s(14,17,20,23,38)=953$. In fact, the sum of nonrepresentable numbers is 
\begin{align*} 
&1 + 2 + 3 + 4 + 5 + 6 + 7 + 8 + 9 + 10 + 11 + 12 + 13 + 15 + 16 + 18 \\
&+ 19 + 21 + 22 + 24 + 25 + 26 + 27 + 29 + 30 + 32 + 33 + 35 + 36 + 39 \\
&+ 41 + 44 + 47 + 49 + 50 + 53 + 64 + 67\\
&=953\,. 
\end{align*} 

By Corollary \ref{cor:a124}, when $a=8,9,10,11$, we have 
\begin{align*} 
s(8,9,10,12)&=104\,,\\
s(9,10,11,13)&=135\,,\\
s(10,11,12,14)&=199\,,\\
s(11,12,13,15)&=272\,. 
\end{align*} 
Indeed, the sums of nonrepresentable numbers are 
\begin{align*} 
&1 + 2 + 3 + 4 + 5 + 6 + 7 + 11 + 13 + 14 + 15 + 23=104\,,\\
&1 + 2 + 3 + 4 + 5 + 6 + 7 + 8 + 12 + 14 + 15 + 16 + 17 + 25=135\,,\\
&1 + 2 + 3 + 4 + 5 + 6 + 7 + 8 + 9 + 13 + 15 + 16 + 17 + 18 + 19\\
&\quad + 27 + 29=199\,,\\
&1 + 2 + 3 + 4 + 5 + 6 + 7 + 8 + 9 + 10 + 14 + 16 + 17 + 18 + 19\\ 
&\quad + 20 + 21 + 29 + 31 + 32=272\,,
\end{align*} 
respectively.

\section{Geometric-like sequence}  

Consider the case 
$$
a_1=a,\, a_2=a+1,\, a_3=a+2, \, a_4=a+2^2,\,\dots,\, a_{k+2}=a+2^k\quad(k\ge 2)\,.
$$ 
Put $a=2^k q+r$ with $0\le r<2^k$. The first line is the sequence $1,2,\dots,2^k-1\pmod a$, that is 
\begin{multline}
a_2=a+1,\, a_3=a+2,\, a_2+a_3=2 a+3, \, a_4=a+4,\,\dots,\\
a_2+a_3+\cdots+a_{k+1}=k a+2^k-1\,.
\label{geo-1st} 
\end{multline} 
The second line is the sequence 
$$
a_{k+2},\,a_{k+2}+a_2,\, a_{k+2}+a_3,\, a_{k+2}+a_2+a_3, \, a_{k+2}+a_4,\,\dots,\,a_{k+2}+a_2+a_3+\cdots+a_{k+1}\,. 
$$
Similarly, the $q$-th line is the sequence 
\begin{multline*}
(q-1)a_{k+2},\,(q-1)a_{k+2}+a_2,\, (q-1)a_{k+2}+a_3,\, (q-1)a_{k+2}+a_2+a_3, \\ 
(q-1)a_{k+2}+a_4,\,\dots,\,(q-1)a_{k+2}+a_2+a_3+\cdots+a_{k+1}\,. 
\end{multline*} 
If $r=0$, then the line is finished. If $r>0$, then the ($q+1$)th line consists from $r$ terms beginning from 
$$
q a_{k+2},\, q a_{k+2}+a_2,\, q a_{k+2}+a_3,\, \cdots\,. 
$$ 

In order to find the sum of the elements, 
consider the exact term which is congruent to $n$ modulo $a$ ($1\le n\le 2^k-1$) in the sequence (\ref{geo-1st}). If $s_1(n)$ is the exponent of $2$ in the canonical representation of $n!$ and $s_2(n)$ is the number of ones in the binary representation of $n$, then 
\begin{align*}  
s_2(n)&=n-s_1(n)\\
&=n-\sum_{i=1}^\infty\fl{\frac{n}{2^i}}=n-\sum_{i=1}^{\fl{\log_2 n}}\fl{\frac{n}{2^i}}
\end{align*}
(see, e.g., \cite[Theorem 3.16]{tat05}).  Hence, the exact expression of the $2^k-1$ terms in the first line is given by 
\begin{multline*} 
a+1=s_2(1)a+1,\,a+2=s_2(2)a+2,\,2 a+3=s_2(3)a+3,\,\dots,\\  
s_2(j)a+j,\,\dots,\,s_2(2^k-1)a+2^k-1\,. 
\end{multline*} 
Since 
\begin{align*}  
&\fl{\frac{2^i}{2^i}}=\cdots=\fl{\frac{2\cdot 2^i-1}{2^i}}=1,\quad 
\fl{\frac{2\cdot 2^i}{2^i}}=\cdots=\fl{\frac{3\cdot 2^i-1}{2^i}}=2\,,\\
&\fl{\frac{3\cdot 2^i}{2^i}}=\cdots=\fl{\frac{4\cdot 2^i-1}{2^i}}=3,\quad \cdots, \\
&\fl{\frac{(2^{k-i}-1)\cdot 2^i}{2^i}}=\cdots=\fl{\frac{2^k-1}{2^i}}=2^{k-i}\,, 
\end{align*}
the sum of all the elements in the first line is equal to 
\begin{align*} 
&a_2+a_3+(a_2+a_3)+\cdots+(a_2+a_3+\cdots+a_{k+1})\\
&=\left(\sum_{j=1}^{2^k-1}s_2(j)\right)a+\sum_{j=1}^{2^k-1}j\\
&=\left(\sum_{j=1}^{2^k-1}j-\sum_{j=1}^{2^k-1}\sum_{i=1}^{\fl{\log_2 j}}\fl{\frac{j}{2^i}}\right)a+2^{k-1}(2^k-1)\\
&=\left(2^{k-1}(2^k-1)-\sum_{i=1}^k\sum_{j=2^i}^{2^k-1}\fl{\frac{j}{2^i}}\right)a+2^{k-1}(2^k-1)\\
&=\left(2^{k-1}(2^k-1)-\sum_{i=1}^k 2^i\bigl(1+2+3+\cdots+(2^{k-i}-1)\bigr)\right)a+2^{k-1}(2^k-1)\\
&=\left(2^{k-1}(2^k-1)-2^{k-1}\sum_{i=1}^k(2^{k-i}-1)\bigr)\right)a+2^{k-1}(2^k-1)\\
&=\left(2^{k-1}(2^k-1)-2^{k-1}\bigl(2^k-k-1)\bigr)\right)a+2^{k-1}(2^k-1)\\ 
&=2^{k-1}k a+2^{k-1}(2^k-1)\,. 
\end{align*}
In general, the sum of all the elements in the $j$-th line ($1\le j\le q$) is  
\begin{align*}  
&(j-1)a_{k+2}2^k+2^{k-1}k a+2^{k-1}(2^k-1)\\
&=(j-1)(a+2^k)2^k+2^{k-1}k a+2^{k-1}(2^k-1)\,. 
\end{align*}
Hence, if $r=0$, then 
\begin{align*}  
\sum_{i=1}^{a-1}m_i&=\sum_{j=1}^q(j-1)(a+2^k)2^k+q 2^{k-1}k a+q 2^{k-1}(2^k-1)\\
&=q(q-1)(a+2^k)2^{k-1}+q 2^{k-1}k a+q 2^{k-1}(2^k-1)\\
&=2^{k-1}q\bigl((q+k-1)a+(2^k q-1)\bigr)\,. 
\end{align*}
Next, by 
\begin{align*} 
\sum_{j=1}^{2^k-1}j s_1(j)&=2^{k-1}(2^k-1)(2^{k+2}-3 k-5)\,,\\
\sum_{j=1}^{2^k-1}\bigl(s_1(j)\bigr)^2&=\frac{2^{k-2}}{3}\bigl(2^{k+1}(2^{k+1}-3 k-6)+(3 k^2+9 k+8)\bigr)\,, 
\end{align*}
the sum of all the square of elements in the first line is  
\begin{align*}  
&a_2^2+a_3^2+(a_2+a_3)^2+\cdots+(a_2+a_3+\cdots+a_{k+1})^2\\
&=2^{k-1}\left(\frac{k(k+1)}{2}a^2+(2^k-1)(k+1)a+\frac{(2^k-1)(2^{k+1}-1)}{3}\right)\,. 
\end{align*} 
The sum of all the square of elements in the $j$-th line ($1\le j\le q$) is  
\begin{align*}  
&\bigl((j-1)a_{k+2}+a_2\bigr)^2+\bigl((j-1)a_{k+2}+a_3\bigr)^2+\cdots+\bigl((j-1)a_{k+2}+a_2+a_3\bigr)^2\\
&\qquad +\cdots+\bigl((j-1)a_{k+2}+a_2+a_3+\cdots+a_{k+1}\bigr)^2\\
&=(j-1)^2 2^k(a+2^k)^2+2(j-1)(a+2^k)\times(\text{sum of the first line})\\
&\qquad +(\text{sum of square of the first line})\\
&=(j-1)^2 2^k(a+2^k)^2\\
&\quad +2(j-1)(a+2^k)\bigl(2^{k-1}k a+2^{k-1}(2^k-1)\bigr)\\
&\quad +2^{k-1}\left(\frac{k(k+1)}{2}a^2+(2^k-1)(k+1)a+\frac{(2^k-1)(2^{k+1}-1)}{3}\right)\,. 
\end{align*} 
Hence, if $r=0$, then 
\begin{align*}  
\sum_{i=1}^{a-1}m_i^2
&=\sum_{j=1}^q(j-1)^2 2^k(a+2^k)^2\\
&\quad +\sum_{j=1}^q 2(j-1)(a+2^k)\bigl(2^{k-1}k a+2^{k-1}(2^k-1)\bigr)\\
&\quad +q 2^{k-1}\left(\frac{k(k+1)}{2}a^2+(2^k-1)(k+1)a+\frac{(2^k-1)(2^{k+1}-1)}{3}\right)\\
&=\frac{q(q-1)(2 q-1)2^{k-1}(a+2^k)^2}{3}\\
&\quad +q(q-1)(a+2^k)\bigl(2^{k-1}k a+2^{k-1}(2^k-1)\bigr)\\
&\quad +q 2^{k-1}\left(\frac{k(k+1)}{2}a^2+(2^k-1)(k+1)a+\frac{(2^k-1)(2^{k+1}-1)}{3}\right)\\
&=\frac{2^{k-2}q}{3}\biggl(
\bigl(3 k(2 q+k-1)+2(q-1)(2 q-1)\bigr)a^2\\ 
&\qquad +2\bigl(2^k(4 q^2+3 k q-3 q+2)-3(q+k)\bigr)a\\
&\qquad +2(2^{2 k+1}q^2-3\cdot 2^k q+1)
\biggr)\,. 
\end{align*}

Next, assume that $r>0$. The exact expression of the $r$ terms in the ($q+1$)-th line is given by 
\begin{multline*} 
q a_{k+2},\,q a_{k+2}+(a+1),\,q a_{k+2}+(a+2),\,q a_{k+2}+(2 a+3),\,\dots\\ 
q a_{k+2}+\bigl(s_2(j)a+j\bigr),\,\dots,\,q a_{k+2}+\bigl(s_2(r-1)a+r-1\bigr)\,.
\end{multline*} 
Thus, by $a-r=2^k q$, the sum of additional terms from the ($q+1$)-th line is 
\begin{align*}  
\mathfrak T_1&:=r q a_{k+2}+\left(\sum_{j=1}^{r-1}s_2(j)\right)a+\sum_{j=1}^{r-1}j\\
&=r q(a+2^k)+\left(\frac{r(r-1)}{2}-\sum_{j=1}^{r-1}s_1(j)\right)a+\frac{r(r-1)}{2}\\
&=\left(\frac{r(r+2 q+1)}{2}-\sum_{j=1}^{r-1}s_1(j)\right)a-\frac{r(r+1)}{2}\,. 
\end{align*}   
The sum of square of additional terms from the ($q+1$)-th line is 
\begin{align*}  
\mathfrak T_2&:=r q^2 a_{k+2}^2+2 q a_{k+2}\sum_{j=1}^{r-1}\bigl(s_2(j)a+j\bigr) +\sum_{j=1}^{r-1}\left(\bigl(s_2(j)\bigr)^2 a^2+2 j s_2(j)a+j^2\right)\\ 
&=r q^2(a+2^k)^2+2 q(a+2^k)\left(\sum_{j=1}^{r-1}s_2(j)\right)a+q(a+2^k)r(r-1)\\
&\quad +\sum_{j=1}^{r-1}\left(j^2-2 j s_1(j)+\bigl(s_1(j)\bigr)^2\right)a^2\\
&\quad +2\left(\sum_{j=1}^{r-1}j \bigl(j-s_1(j)\bigr)\right)a+\frac{r(r-1)(2 r-1)}{6}\\ 
&=r\bigl((q+1)a-r\bigr)^2+2\bigl((q+1)a-r\bigr)\left(\frac{r(r-1)}{2}-\sum_{j=1}^{r-1}s_1(j)\right)a\\
&\quad +\bigl((q+1)a-r\bigr)r(r-1)\\
&\quad +\left(\frac{r(r-1)(2 r-1)}{6}-2\sum_{j=1}^{r-1}j s_1(j)+\sum_{j=1}^{r-1}\bigl(s_1(j)\bigr)^2\right)a^2\\ 
&\quad +2\left(\frac{r(r-1)(2 r-1)}{6}-\sum_{j=1}^{r-1}j s_1(j)\right)a+\frac{r(r-1)(2 r-1)}{6}\\
&=\biggl(\sum_{j=1}^{r-1}\bigl((s_1(j))^2-2 j s_1(j)-2(q+1)s_1(j)\bigr)\\ 
&\qquad +\frac{r\bigl(6q(q+r+1)+(r+1)(2 r+1)\bigr)}{6}\biggr)a^2\\
&\quad+\biggl(2\sum_{j=1}^{r-1}(r-j)s_1(j)-\frac{r(r+1)(3 q+r+2)}{3}\biggr)a 
 +\frac{r(r+1)(2 r+1)}{6}\,.  
\end{align*}

By using the third formula in Lemma \ref{lem1}, we get 
\begin{align*}
&s(a,a+1,a+2,a+2^2,\dots,a+2^k)\\
&=\frac{1}{2 a}\biggl(\frac{2^{k-2}q}{3}\biggl(
\bigl(3 k(2 q+k-1)+2(q-1)(2 q-1)\bigr)a^2\\ 
&\qquad +2\bigl(2^k(4 q^2+3 k q-3 q+2)-3(q+k)\bigr)a\\
&\qquad +2(2^{2 k+1}q^2-3\cdot 2^k q+1)+\mathfrak T_2\biggr)\\
&\quad-\frac{1}{2}\biggl(2^{k-1}q\bigl((q+k-1)a-(2^k q-1)\bigr)+\mathfrak T_1\biggr)
+\frac{a^2-1}{12}\\
&=\frac{2^{k-3}q}{3}\bigl(4(q-1)(q-2)+3 k(2 q+k-3)\bigr)a\\
&\quad+\frac{1}{12}\bigl(2^{2 k}q(4 q^2+3 k q-3 q+2)-3\cdot 2^k q(q+k-2)-6\bigr)\\ 
&\quad+\frac{1}{a}(2^{2 k+1}q^2-3\cdot 2^k q+1)-\frac{\mathfrak T_1}{2}+\frac{\mathfrak T_2}{2 a}+\frac{a^2-1}{12}\,. 
\end{align*}

\begin{theorem} 
For integers $k\ge 2$, $q$ and $r$, satisfying $a=2^k q+r$ with $0\le r<2^k$, 
\begin{align*}
&s(a,a+1,a+2,a+2^2,\dots,a+2^k)\\
&=\frac{2^{k-3}q}{3}\bigl(4(q-1)(q-2)+3 k(2 q+k-3)\bigr)a\\
&\quad+\frac{1}{12}\bigl(2^{2 k}q(4 q^2+3 k q-3 q+2)-3\cdot 2^k q(q+k-2)-6\bigr)\\ 
&\quad+\frac{1}{a}(2^{2 k+1}q^2-3\cdot 2^k q+1)-\frac{\mathfrak T_1}{2}+\frac{\mathfrak T_2}{2 a}+\frac{a^2-1}{12}\,, 
\end{align*}
where 
$$  
\mathfrak T_1=\left(\frac{r(r+2 q+1)}{2}-\sum_{j=1}^{r-1}s_1(j)\right)a-\frac{r(r+1)}{2}
$$ 
and 
\begin{align*}  
\mathfrak T_2&=\biggl(\sum_{j=1}^{r-1}\bigl((s_1(j))^2-2 j s_1(j)-2(q+1)s_1(j)\bigr)\\ 
&\qquad +\frac{r\bigl(6q(q+r+1)+(r+1)(2 r+1)\bigr)}{6}\biggr)a^2\\
&\quad+\biggl(2\sum_{j=1}^{r-1}(r-j)s_1(j)-\frac{r(r+1)(3 q+r+2)}{3}\biggr)a 
 +\frac{r(r+1)(2 r+1)}{6}\,.  
\end{align*}   
\label{th:geom} 
\end{theorem}

\subsection{Example}  

Consider the sequence $16,17,18,20,24$. Then, $a=16$, $k=3$ $q=2$ and $r=0$. By Theorem \ref{th-aa}, we have $s(16,17,18,20,24)=684$. In fact, the sum of nonrepresentable numbers is 
\begin{align*} 
&1+2+3+4+5+6+7+8+9+10+11+12+13+14+15+19\\
&+21+22+23+25+26+27+28+29+30+31+39+43+45\\
&+46+47+63\\
&=684\,. 
\end{align*} 

Consider the sequence $25,26,27,29,33$. Then, $a=25$, $k=3$ $q=3$ and $r=1$. By Theorem \ref{th-aa}, we have $s(25,26,27,29,33)=2557$. If the sequence is $25,26,27,29,33,41$, then $a=25$, $k=4$ $q=1$ and $r=9$. By Theorem \ref{th-aa}, we have $s(25,26,27,29,33,41)=1827$.

\section{Conclusion} 

The classical purpose of the Frobenius problem is to find the largest integer (Frobenius number) which is not represented in terms of $a_1,a_2,\dots,a_k$. Then, as related concepts, the number of nonrepresentable positive integers (Sylvester number) and the sum of nonrepresentable positive integers (Sylvester sum) have been also studied. Recently, as an extended concept, the weighted sum is introduced. When $k\ge 3$, it is very difficult to find any explicit expression of these numbers and/or sums. However, when $a_1,a_2,\dots,a_k$ forms some special sequences or patterns, it may be possible to find any expressions. In this paper, we give explicit expressions of the Sylvester sum ($\lambda=1$) and the weighed sum ($\lambda\ne 1$), where $a_1,a_2,\dots,a_k$ forms arithmetic progressions. As applications, various other cases are also considered, including weighted sums, almost arithmetic sequences, arithmetic sequences with an additional term, and geometric-like sequences. Several examples illustrate and confirm our results.

\end{document}